\documentclass[11pt]{article}
    
\usepackage{amsthm}
\usepackage{amssymb}
\usepackage{hyperref}
\usepackage{mathrsfs}
\usepackage{color} 
\usepackage{amsmath}
\usepackage{diagbox} 

\def\fl#1{\left\lfloor#1\right\rfloor}

\def\stif#1#2{\left[#1\atop#2\right]}
\def\sts#1#2{\left\{#1\atop#2\right\}}
\def\sttf2#1#2{\left[\!\!\left[#1\atop#2\right]\!\!\right]}
\def\stf3f#1#2{\left[\!\!\left[\!\!\left[#1\atop#2\right]\!\!\right]\!\!\right]}
\def\stff4#1#2{\left[\!\!\left[\!\!\left[\!\!\left[#1\atop#2\right]\!\!\right]\!\!\right]\!\!\right]}
\def\stss2#1#2{\left\{\!\!\left\{#1\atop#2\right\}\!\!\right\}}

\newtheorem{theorem}{Theorem}

\def \Sg{\mathfrak S} 

\allowdisplaybreaks

\begin{document} 

\title{Stirling numbers with level $2$ and poly-Bernoulli numbers with level $2$}  

\author{ 
Takao Komatsu
\\ 
\small Department of Mathematical Sciences, School of Science\\
\small Zhejiang Sci-Tech University\\
\small Hangzhou 310018 China\\
\small \texttt{komatsu@zstu.edu.cn}
}

\date{
}

\maketitle 

\begin{abstract}  
In this paper, we introduce poly-Bernoulli numbers with level $2$, related to the Stirling numbers of the second kind with level $2$, and study several properties of poly-Bernoulli numbers with level $2$ from their expressions, relations, and congruences. Poly-Bernoulli numbers with level $2$ have strong connections with poly-Cauchy numbers with level $2$.  
In a special case, we can determine the denominators of Bernoulli numbers with level $2$ by showing a von Staudt-Clausen like theorem.    


\noindent 
{\bf MR Subject Classifications:} 11B73, 05A15, 05A19, 11A07, 11B37, 11B68, 11B75.
\end{abstract}

\section{Introduction} 

Let $\Sg_n$ denote the set of permutations of the set  $[n]:=\{1, 2, \dots, n\}$. 
For $n, k \geq 0$, let  $\Sg_{(n,k)}$ denote the set of permutations of $\Sg_n$  having exactly $k$ cycles, satisfying  $\Sg_n=\cup_{k=0}^n\Sg_{(n,k)}$. Let $s$ be a positive integer.  The Stirling numbers of the first kind with level $s$, denoted by $\sttf2{n}{k}_s$ (\cite{KRV1}), are defined as the number of ordered $s$-tuples $(\sigma_1, \sigma_2, \dots, \sigma_s)\in \Sg_{(n,k)}\times \Sg_{(n,k)} \times \cdots \times \Sg_{(n,k)}=\Sg_{(n,k)}^s$, such that
$$ 
\min(\sigma_1)=\min(\sigma_2)=\cdots=\min(\sigma_s)\,. 
$$  
The Stirling numbers of the first kind with higher level satisfies the recurrence relation
$$ 
\sttf2{n}{k}_s=\sttf2{n-1}{k-1}_s+(n-1)^s\sttf2{n-1}{k} 
$$ 
with the initial conditions $\sttf2{0}{0}_s=1$ and $\sttf2{n}{0}_s=\sttf2{0}{n}_s=0$ ($n\geq 1$).  
The Stirling numbers of the first kind with higher level are yielded from the coefficients of the polynomial as 
$$
x(x+1^s)(x+2^s)\cdots(x+(n-1)^s)=\sum_{k=0}^n\sttf2{n}{k}_s x^k\,.
$$ 

When $s=1$, 
$$
\stif{n}{k}=\sttf2{n}{k}_1=|\Sg_{(n,k)}|
$$ 
are the original (unsigned) Stirling numbers of the first kind.  
When $s=2$, the Stirling numbers of the first kind with level $2$ (\cite{Ko20}) are related with the central factorial numbers of the first kind $t(n,k)$ (\cite{BSSV}) as $\sttf2{n}{k}_2=t(2 n, 2 k)$.    
Notice that the original Stirling numbers of the first kind and the Stirling numbers of the first kind with level $2$ are used to express poly-Cauchy numbers (\cite{Ko1}) and poly-Cauchy numbers with level $2$ (\cite{Ko20,KP}), respectively.   


On the other hand, for $n, k \geq 0$, let  $\Pi_{(n,k)}$ denote the set of all partitions of $[n]$  having exactly $k$ non-empty blocks. 
Given a partition $\pi$ in $\Pi_n$, let $\min(\pi)$ denote the set of the minimal elements in each block of $\pi$.  For a positive integer $s$, the Stirling numbers of the second kind with level $s$, denoted by $\stss2{n}{k}_s$ (\cite{KRV2}), are defined as the number of ordered $s$-tuples $(\pi_1, \pi_2, \dots, \pi_s)\in \Pi_{(n,k)}\times \Pi_{(n,k)} \times \cdots \times \Pi_{(n,k)}=\Pi_{(n,k)}^s$, such that
\begin{align}\label{mincond}
\min(\pi_1)=\min(\pi_2)=\cdots=\min(\pi_s)\,.
\end{align}
The Stirling numbers of the second kind with higher level satisfies the recurrence relation
$$ 
\stss2{n}{k}_s=\stss2{n-1}{k-1}_s+k^s\stss2{n-1}{k}_s\,.
$$ 
with the initial conditions $\stss2{0}{0}_s=1$ and $\stss2{n}{0}_s=\stss2{0}{n}_s=0$ ($n\geq 1$).  
The Stirling numbers of the second kind with higher level are yielded from the coefficients of the polynomial as 
\begin{equation}
x^n=\sum_{k=0}^n\stss2{n}{k}_s x(x-1^s)(x-2^s)\cdots\bigl(x-(k-1)^s\bigr)\,.
\label{def:stss}
\end{equation}

When $s=1$, 
$$
\sts{n}{k}=\stss2{n}{k}_1
$$ 
are the original Stirling numbers of the second kind.  
When $s=2$, we have that $\stss2{n}{k}_2=T(2 n, 2 k)$, where $T(n, k)$ are the central factorial numbers of the second kind ({\it cf.} \cite{BSSV,CR}), satisfying 
\begin{equation}
x^n=\sum_{k=0}^n T(n,k) x\bigl(x+\frac{k}{2}-1\bigr)\bigl(x+\frac{k}{2}-2\bigr)\cdots\bigl(x-\frac{k}{2}+1\bigr)\,.
\end{equation}

As the original Stirling numbers of the second kind are used to express poly-Bernoulli numbers explicitly (\cite{Kaneko}), we intend to introduce poly-Bernoulli numbers with level $2$, related to the Stirling numbers of the second kind with level $2$.  
In this paper, we study several properties of poly-Bernoulli numbers with level $2$ from their expressions, relations, and congruences. Poly-Bernoulli numbers with level $2$ have strong connections with poly-Cauchy numbers with level $2$.  
In a special case, we can determine the denominators of Bernoulli numbers with level $2$ by showing a von Staudt-Clausen like theorem.

\section{Some expressions}   

In \cite{KRV2}, the Stirling numbers of the second kind with higher level are expressed explicitly as 
\begin{equation}
\stss2{n}{k}_s=\sum_{j=1}^k\frac{j^{n s}}{\prod_{i=0, i\neq j}^k(j^s-i^s)}\quad(1\le k\le n)\,.  
\label{ex:10} 
\end{equation}   
When $s=1$, by 
\begin{equation}
(k-j)!j!=(-1)^{k-j}\prod_{i=0\atop i\neq j}^k(j-i)\,, 
\label{ex:20} 
\end{equation} 
this is reduced to a famous expression of the original Stirling numbers of the second kind:
\begin{equation}
\sts{n}{k}=\frac{1}{k!}\sum_{j=1}^k(-1)^{k-j}\binom{k}{k-j}j^{n}\quad(1\le k\le n)\,.   
\label{ex:30} 
\end{equation}  
When $s=2$, (\ref{ex:10}) is reduced to an expression 
\begin{equation}
\stss2{n}{k}_2=\frac{2}{(2 k)!}\sum_{j=1}^k(-1)^{k-j}\binom{2 k}{k-j}j^{2 n}\quad(1\le k\le n)   
\label{ex:32} 
\end{equation}  
(\cite[Proposition 2.4 (xiii)]{BSSV}). However, no explicit expression for $s\ge 3$ has not been found yet. It implies that unfortunately, 
$$ 
\stss2{n}{k}_3\ne \frac{3}{(3 k)!}\sum_{j=1}^k(-1)^{k-j}\binom{3 k}{k-j}j^{3 n}   
$$ 
or something like this.  

There exists a different explicit expression from (\ref{ex:10}). 

\begin{theorem}  
For integers $n$ and $k$ with $2\le k\le n$, 
$$ 
\stss2{n}{k}_s=\sum_{j=1}^{k-1}\frac{j^{(k-1)s}(j^{(n-k+1)s}-k^{(n-k+1)s})}{\prod_{i=0, i\neq j}^k(j^s-i^s)}  
$$ 
with $\stss2{n}{1}_s=1$ ($n\ge 1$). 
\label{th:100} 
\end{theorem}  

\noindent 
{\it Remark.}  
When $s=1$ in Theorem \ref{th:100}, by (\ref{ex:20}) and 
$$ 
\sum_{j=1}^k(-1)^j\binom{k}{j}j^{k-1}=0\,, 
$$  
we have (\ref{ex:30}) again.  

In \cite{KRV2}, the ordinary generating function of the Stirling numbers of the second kind with higher level is given by 
$$ 
\sum_{n=k}^\infty\stss2{n}{k}_s x^n=\frac{x^k}{(1-x)(1-2^s x)\cdots(1-k^s x)}\,. 
$$ 

The exponential generating function of the Stirling numbers of the second kind with higher level is given as follows.   

\begin{theorem}  
For $k\ge 1$, 
$$
\sum_{n=k}^\infty\stss2{n}{k}_s\frac{x^n}{n!}=\sum_{j=1}^k\frac{e^{j^s x}}{\prod_{i=1, i\ne j}^k(j^s-i^s)}+\frac{(-1)^k}{(k!)^s}\,. 
$$ 
\label{th:200}  
\end{theorem}  

\noindent 
{\it Remark.}  
1) When $s=1$ in Theorem \ref{th:200}, we have the exponential generating function of the original Stirling numbers of the second kind: 
\begin{align*}  
\sum_{n=k}^\infty\sts{n}{k}\frac{x^n}{n!}&=\frac{1}{k!}\left(\sum_{j=1}^k\frac{k!}{(k-j)!j!}(e^x)^j(-1)^{k-j}+(-1)^k\right)\\
&=\frac{(e^x-1)^k}{k!}\,. 
\end{align*}   

\noindent
2) Another variation is similarly shown as follows.  
$$
\sum_{n=k}^\infty\stss2{n}{k}_s\frac{x^{n s}}{(n s)!}=\frac{1}{s}\sum_{\ell=0}^{s-1}\sum_{j=1}^k\frac{e^{\zeta^\ell j x}}{\prod_{i=1, i\ne j}^k(j^s-i^s)}+\frac{(-1)^k}{(k!)^s}\,. 
$$  
where $\zeta:=e^{2\pi i/s}$ is the $s$-th root of unity.  

\begin{proof}[Proof of Theorem \ref{th:200}] 
\par  
1)  
By (\ref{ex:10}), we have 
\begin{align*}
\sum_{n=k}^\infty\stss2{n}{k}_s\frac{x^n}{n!}
&=\sum_{n=k}^\infty\sum_{j=1}^k\frac{j^{n s}}{\prod_{i=0, i\neq j}^k(j^s-i^s)}\frac{x^n}{n!}\\  
&=\sum_{j=1}^k\frac{1}{\prod_{i=0, i\neq j}^k(j^s-i^s)}\sum_{n=0}^\infty j^{n s}\frac{x^n}{n!}\\
&\quad -\sum_{j=1}^k\frac{1}{\prod_{i=0, i\neq j}^k(j^s-i^s)}\sum_{n=0}^{k-1}j^{n s}\frac{x^n}{n!}\\
&=\sum_{j=1}^k\frac{e^{j^s x}}{\prod_{i=1, i\ne j}^k(j^s-i^s)}+\frac{(-1)^k}{(k!)^s}\,. 
\end{align*} 
Here,  
$$
\sum_{j=1}^k\frac{1}{\prod_{i=0, i\neq j}^k(j^s-i^s)}=\frac{(-1)^{k-1}}{(k!)^s}
$$ 
and
$$  
\sum_{j=1}^k\frac{j^{n s}}{\prod_{i=0, i\neq j}^k(j^s-i^s)}=0\quad(1\le n\le k-1)\,.
$$ 
\par\noindent   
2) By Theorem \ref{th:100}, we have  
\begin{align*}
\sum_{n=k}^\infty\stss2{n}{k}_s\frac{x^n}{n!} 
&=\sum_{n=k}^\infty\sum_{j=1}^{k-1}\frac{j^{(k-1)s}(j^{(n-k+1)s}-k^{(n-k+1)s})}{\prod_{i=0, i\neq j}^k(j^s-i^s)}\\  
&=\sum_{j=1}^k\frac{1}{\prod_{i=0, i\neq j}^k(j^s-i^s)}\sum_{n=0}^\infty j^{n s}\frac{x^n}{n!}\\
&\quad -\sum_{j=1}^k\frac{1}{\prod_{i=0, i\neq j}^k(j^s-i^s)}\sum_{n=0}^\infty\left(\frac{j}{k}\right)^{(k-1)s}k^{n s}\frac{x^n}{n!}\\
&\quad -\sum_{j=1}^k\frac{1}{\prod_{i=0, i\neq j}^k(j^s-i^s)}\sum_{n=0}^{k-1}j^{n s}\frac{x^n}{n!}\\
&\quad +\sum_{j=1}^k\frac{1}{\prod_{i=0, i\neq j}^k(j^s-i^s)}\sum_{n=0}^{k-1}\left(\frac{j}{k}\right)^{(k-1)s}k^{n s}\frac{x^n}{n!}\\
&=\sum_{j=1}^k\frac{e^{j^s x}}{\prod_{i=1, i\ne j}^k(j^s-i^s)}-0\\
&-\sum_{j=1}^k\frac{1}{\prod_{i=0, i\neq j}^k(j^s-i^s)}\sum_{n=0}^{k-1}j^{n s}\frac{x^n}{n!}+0\\
&=\sum_{j=1}^k\frac{e^{j^s x}}{\prod_{i=1, i\ne j}^k(j^s-i^s)}+\frac{(-1)^k}{(k!)^s}\,. 
\end{align*} 
\end{proof}

\section{Poly-Bernoulli numbers with level $2$} 

Poly-Bernoulli numbers are defined by the generating function 
\begin{equation}  
\frac{{\rm Li}_k(1-e^{-x})}{1-e^{-x}}=\sum_{n=0}^\infty B_n^{(k)}\frac{x^n}{n!}
\label{def:pber}
\end{equation}
(\cite{Kaneko}), where 
$$
{\rm Li}_k(z)=\sum_{n=1}^\infty\frac{z^n}{n^k}
$$ 
is the polylogarithm function. 

When $k=1$, it is reduced to the generating function of the original Bernoulli numbers:  
$$
\frac{x}{1-e^{-x}}=\sum_{n=0}^\infty B_n^{(1)}\frac{x^n}{n!} 
$$ 
with $B_1^{(1)}=1/2$. Another definition is given by 
\begin{equation}
\frac{x}{e^{x}-1}=\sum_{n=0}^\infty B_n\frac{x^n}{n!} 
\label{def:ber}
\end{equation} 
with $B_1=-1/2$. 
 
Then, poly-Bernoulli numbers can be expressed explicitly in terms of the Stirling numbers of the second kind (\cite[Theorem 1]{Kaneko}):   
\begin{equation}  
B_n^{(k)}=\sum_{m=0}^n\sts{n}{m}\frac{(-1)^{n-m}m!}{(m+1)^k}\,. 
\label{exp:pber}
\end{equation}
There have been many generalizations of Bernoulli or poly-Bernoulli numbers. In this paper, we introduced poly-Bernoulli numbers with level $2$ by using higher-level Stirling numbers (\cite{KRV1,KRV2}). 

In \cite{KP}, {\it poly-Cauchy numbers}  $\mathfrak C_n^{(k)}$ {\it with level $2$} are defined by  
\begin{equation}  
{\rm Lif}_{2,k}({\rm arcsinh}x)=\sum_{n=0}^\infty\mathfrak C_n^{(k)}\frac{x^n}{n!}\,,  
\label{def:pcaulevel2}  
\end{equation}  
where  ${\rm arcsinh}x$ is the inverse hyperbolic sine function and  
$$
{\rm Lif}_{2,k}(z)=\sum_{m=0}^\infty\frac{z^{2 m}}{(2 m)!(2 m+1)^k}\,. 
$$ 
The function ${\rm Lif}_{2,k}(z)$ is an analogue of {\it polylogarithm factorial} or {\it polyfactorial} function ${\rm Lif}_k(z)$ \cite{Ko1,Ko2}, defined by  
$$
{\rm Lif}_k(z)=\sum_{m=0}^\infty\frac{z^{m}}{m!(m+1)^k}\,. 
$$ 
By using the polyfactorial function, poly-Cauchy numbers (of the first kind) $c_n^{(k)}$ are defined as 
\begin{equation}
{\rm Lif}_k\bigl(\log(1+x)\bigr)=\sum_{n=0}^\infty c_n^{(k)}\frac{x^n}{n!}\,.
\label{def:pcau} 
\end{equation} 
When $k=1$, by ${\rm Lif}_1(z)=(e^z-1)/z$, $c_n=c_n^{(1)}$ are the original Cauchy numbers defined by  
$$
\frac{x}{\log(1+x)}=\sum_{n=0}^\infty c_n\frac{x^n}{n!}\,. 
$$ 

Define the {\it polylogarithm function} ${\rm Li}_{2,k}(z)$ {\it with level $2$} by 
\begin{equation}  
{\rm Li}_{k,2}(z)=\sum_{n=0}^\infty\frac{z^{2 n+1}}{(2 n+1)^k}\,. 
\label{def:genpolylog}
\end{equation} 
Then, {\it poly-Bernoulli numbers}  $\mathfrak B_n^{(k)}$ {\it with level $2$} are defined by  
\begin{equation}  
\frac{{\rm Li}_{2,k}\bigl(2\sin(x/2)\bigr)}{2\sin(x/2)}=\sum_{n=0}^\infty\mathfrak B_n^{(k)}\frac{x^n}{n!}\,.   
\label{def:pbel2}  
\end{equation}  
Note that $\mathfrak B_n^{(k)}=0$ for odd $n$. 

The generating function of the poly-Cauchy numbers with level $2$ can be written in the form of iterated integrals (\cite[Theorem 2.1]{KP}): 
\begin{multline*} 
\frac{1}{{\rm arcsinh}x}\underbrace{\int_0^x\frac{1}{{\rm arcsinh}x\sqrt{1+x^2}}\cdots\int_0^x\frac{1}{{\rm arcsinh}x\sqrt{1+x^2}}}_{k-1}\times x\underbrace{d x\cdots d x}_{k-1}\\
=\sum_{n=0}^\infty\mathfrak C_n^{(k)}\frac{x^n}{n!}\quad(k\ge 1)\,. 
\end{multline*} 
We can also write the generating function of the poly-Bernoulli numbers with level $2$ in (\ref{def:pbel2}) in the form of iterated integrals. 

\begin{theorem} 
For $k\ge 1$, we have 
\begin{multline*} 
\frac{1}{2\sin\frac{x}{2}}\underbrace{\int_0^x\frac{1}{2\tan\frac{x}{2}}\cdots\int_0^x\frac{1}{2\tan\frac{x}{2}}}_{k-1}\times \frac{1}{2}\log\frac{1+2\sin\frac{x}{2}}{1-2\sin\frac{x}{2}}\underbrace{d x\cdots d x}_{k-1}\\
=\sum_{n=0}^\infty\mathfrak B_n^{(k)}\frac{x^n}{n!}\,. 
\end{multline*} 
\label{th:iterated}
\end{theorem}  
\begin{proof} 
Since 
$$ 
\frac{d}{d z}{\rm Li}_{2,k}(z)=\frac{1}{z}{\rm Li}_{2,k-1}(z)\,, 
$$ 
we have 
\begin{align*} 
{\rm Li}_{2,k}(z)&=\int_0^z\frac{{\rm Li}_{2,k-1}(z_1)}{z_1}d z_1\\
&=\int_0^z\frac{d z_1}{z_1}\int_0^{z_1}\frac{{\rm Li}_{2,k-2}(z_2)}{z_2}d z_2\\
&=\int_0^z\frac{d z_1}{z_1}\int_0^{z_1}\frac{d z_2}{z_2}\cdots\int_0^{z_{k-2}}\frac{{\rm Li}_{2,1}(z_{k-1})}{z_{k-1}}d z_{k-1}\\ 
&=\int_0^z\frac{d z_1}{z_1}\int_0^{z_1}\frac{d z_2}{z_2}\cdots\int_0^{z_{k-2}}\frac{1}{z_{k-1}}\frac{1}{2}\log\frac{1+z_{k-1}}{1-z_{k-1}}d z_{k-1}\,. 
\end{align*}
Putting $z=z_1=\cdots=z_{k-1}=2\sin(x/2)$, we get 
\begin{multline*} 
\frac{{\rm Li}_{2,k}\bigl(2\sin(x/2)\bigr)}{2\sin(x/2)}=\frac{1}{2\sin(x/2)}\int_0^x\frac{\cos(x/2)}{2\sin(x/2)}d x\int_0^x\frac{\cos(x/2)}{2\sin(x/2)}d x\\
\cdots\int_0^x\frac{\cos(x/2)}{2\sin(x/2)}\frac{1}{2}\log\frac{1+2\sin(x/2)}{1-2\sin(x/2)}d x\,. 
\end{multline*} 
\end{proof}

\section{Explicit formulae and recurrence relations} 

From the definition in (\ref{def:pbel2}), we see that 
\begin{align*}
\mathfrak B_0^{(k)}&=1\,,\\
\mathfrak B_2^{(k)}&=\frac{2}{3^k}\,,\\ 
\mathfrak B_4^{(k)}&=-\frac{2}{3^k}+\frac{24}{5^k}\,,\\ 
\mathfrak B_6^{(k)}&=\frac{2}{3^k}-\frac{120}{5^k}+\frac{720}{7^k}\,,\\ 
\mathfrak B_8^{(k)}&=-\frac{2}{3^k}+\frac{504}{5^k}-\frac{10080}{7^k}+\frac{40320}{9^k}\,,\\ 
\mathfrak B_{10}^{(k)}&=\frac{2}{3^k}-\frac{2040}{5^k}+\frac{105840}{7^k}-\frac{32659200}{9^k}+\frac{3628800}{11^k}\,. 
\end{align*}
In this section, we shall show some explicit formulae and some recurrence relations.  

Poly-Cauchy numbers with level $2$ can be expressed explicitly in terms of the Stirling numbers of the second kind with level $2$ (\cite[Theorem 1]{Ko20}): 
$$
\mathfrak C_{2 n}^{(k)}=\sum_{m=0}^n\sttf2{n}{m}_2\frac{(-4)^{n-m}}{(2 m+1)^k}\,. 
$$ 
Poly-Bernoulli numbers with level $2$ can be expressed explicitly in terms of the Stirling numbers of the second kind with level $2$.  It is a natural extension of the expression in (\ref{exp:pber}).  

\begin{theorem}  
For $n\ge 0$, 
$$ 
\mathfrak B_{2 n}^{(k)}=\sum_{m=0}^n\stss2{n}{m}_2\frac{(-1)^{n-m}(2 m)!}{(2 m+1)^k}\,. 
$$ 
\label{th1}
\end{theorem} 
\begin{proof}  
We use the power series of powers of trigonometric functions 
$$
\left(2\sin\frac{x}{2}\right)^{2 m}=\sum_{n=m}^\infty(-1)^{n-m}\frac{(2 m)!}{(2 n)!}\stss2{n}{m}_2 x^{2 n} 
$$ 
(see \cite[Theorem 4.1.1 (4.1.1)]{BSSV}). Then, by (\ref{def:genpolylog}) and (\ref{def:pbel2}), we have 
\begin{align*}  
\sum_{n=0}^\infty\mathfrak B_n^{(k)}\frac{x^n}{n!}&=\sum_{n=0}^\infty\mathfrak B_{2 n}^{(k)}\frac{x^{2 n}}{(2 n)!}\\ 
&=\sum_{m=0}^\infty\frac{\bigl(2\sin(x/2)\bigr)^{2 m}}{(2 m+1)^k}\\
&=\sum_{m=0}^\infty\frac{1}{(2 m+1)^k}\sum_{n=m}^\infty(-1)^{n-m}\frac{(2 m)!}{(2 n)!}\stss2{n}{m}_2 x^{2 n}\\  
&=\sum_{n=0}^\infty\sum_{m=0}^n\stss2{n}{m}_2\frac{(-1)^{n-m}(2 m)!}{(2 m+1)^k}\frac{x^{2 n}}{(2 n)!}\,. 
\end{align*}
Comparing the coefficients on both sides, we get the desired result.  
\end{proof}

Next, we shall show an explicit formula without Stirling numbers. 

\begin{theorem}  
For integers $n$ and $k$ with $n\ge 0$, 
$$ 
\mathfrak B_{2 n}^{(k)}=\sum_{m=0}^n\frac{1}{(2 m+1)^k}\sum_{i_1+\cdots+i_{2 m}=n-m\atop i_1,\dots,i_{2 m}\ge 0}\left(-\frac{1}{4}\right)^{n-m}\binom{2 n}{2 i_1+1,\cdots,2 i_{2 m}+1}\,, 
$$ 
where 
$$
\binom{2 n}{2 i_1+1,\cdots,2 i_{2 m}+1}=\frac{\bigl((2 i_1+1)+\cdots+(2 i_{2 m}+1)\bigr)!}{(2 i_1+1)!\cdots(2 i_{2 m}+1)!}
$$ 
is the multinomial coefficient. 
\label{th:exp}
\end{theorem}  
\begin{proof}  
Since 
$$
2\sin\frac{x}{2}=\sum_{\ell=0}^\infty\frac{(-1)^\ell}{(2\ell+1)!}\frac{x^{2\ell+1}}{2^{2\ell}}\,, 
$$ 
we have 
\begin{align*}  
&\sum_{n=0}^\infty\mathfrak B_{2 n}^{(k)}\frac{x^{2 n}}{(2 n)!}\\
&=\sum_{m=0}^\infty\frac{1}{(2 m+1)^k}\left(2\sin\frac{x}{2}\right)^{2 m}\\ 
&=\sum_{m=0}^\infty\frac{1}{(2 m+1)^k}\left(\sum_{\ell=0}^\infty\frac{(-1)^\ell}{(2\ell+1)!}\frac{x^{2\ell+1}}{2^{2\ell}}\right)^{2 m}\\ 
&=\sum_{m=0}^\infty\frac{1}{(2 m+1)^k}\sum_{n=m}^\infty\sum_{i_1+\cdots+i_{2 m}=n-m\atop i_1,\dots,i_{2 m}\ge 0}\frac{(-1)^{i_1+\cdots+i_{2 m}}}{(2 i_1+1)!\cdots(2 i_{2 m}+1)!}\frac{x^{2 n}}{2^{2 i_1+\cdots+2 i_{2 m}}}\\
&=\sum_{n=0}^\infty\sum_{m=0}^n\frac{1}{(2 m+1)^k}\sum_{i_1+\cdots+i_{2 m}=n-m\atop i_1,\dots,i_{2 m}\ge 0}\frac{(-1)^{n-m}x^{2 n}}{(2 i_1+1)!\cdots(2 i_{2 m}+1)!2^{2(n-m)}}\,.
\end{align*}
Comparing the coefficients on both sides, we get the desired result.  
\end{proof}

There exists a recurrence formula for $\mathfrak B_n^{(k)}$ in terms of $\mathfrak B_n^{(k-1)}$ and the original Bernoulli numbers $B_n$ in (\ref{def:ber}). In fact, $B_{n}=B_{n}^{(1)}$ for even $n$.  

\begin{theorem}  
For integers $n$ and $k$ with $n\ge 0$ and $k\ge 1$,  
$$
\mathfrak B_{2 n}^{(k-1)}=\mathfrak B_{2 n}^{(k)}+(2 n)!\sum_{m=0}^{n-1}\frac{4\bigl((-1)^{n-m}-(-4)^{n-m}\bigr)B_{2 n-2 m}\mathfrak B_{2 m+2}^{(k)}}{(2 n-2 m)!(2 m+1)!}\,. 
$$ 
\label{th:relation-bbb} 
\end{theorem}  
\begin{proof}  
From the definition in (\ref{def:pbel2}), we see  
$$
{\rm Li}_{2,k}\bigl(2\sin(x/2)\bigr)=2\sin(x/2)\sum_{n=0}^\infty\mathfrak B_{2 n}^{(k)}\frac{x^{2 n}}{(2 n)!}\,.   
$$ 
Differentiating both sides by $x$, we have 
\begin{align*}  
&\frac{\cos(x/2)}{2\sin(x/2)}{\rm Li}_{2,k-1}\bigl(2\sin(x/2)\bigr)\\
&=\cos\frac{x}{2}\sum_{n=0}^\infty\mathfrak B_{2 n}^{(k)}\frac{x^{2 n}}{(2 n)!}+2\sin\frac{x}{2}\sum_{n=1}^\infty\mathfrak B_{2 n}^{(k)}\frac{x^{2 n-1}}{(2 n-1)!}\,. 
\end{align*}
Hence, 
\begin{equation}
\sum_{n=0}^\infty\mathfrak B_{2 n}^{(k-1)}\frac{x^{2 n}}{(2 n)!}=\sum_{n=0}^\infty\mathfrak B_{2 n}^{(k)}\frac{x^{2 n}}{(2 n)!}+2\tan\frac{x}{2}\sum_{n=0}^\infty\mathfrak B_{2 n+2}^{(k)}\frac{x^{2 n+1}}{(2 n+1)!}\,. 
\label{eq:200} 
\end{equation} 
Since 
\begin{align*} 
&2\tan\frac{x}{2}\sum_{n=0}^\infty\mathfrak B_{2 n+2}^{(k)}\frac{x^{2 n+1}}{(2 n+1)!}\\
&=\left(\sum_{n=1}^\infty\frac{4\bigl((-1)^n-(-4)^n\bigr)B_{2 n}}{(2 n)!}x^{2 n-1}\right)\left(\sum_{n=0}^\infty\mathfrak B_{2 n+2}^{(k)}\frac{x^{2 n+1}}{(2 n+1)!}\right)\\ 
&=\sum_{n=1}^\infty\sum_{m=0}^{n-1}\frac{4\bigl((-1)^{n-m}-(-4)^{n-m}\bigr)B_{2 n-2 m}\mathfrak B_{2 m+2}^{(k)}}{(2 n-2 m)!(2 m+1)!}x^{2 n}\,, 
\end{align*}
comparing the coefficients of both sides of (\ref{eq:200}), we get the desired result.  
\end{proof}

\section{Relations with poly-Cauchy numbers with level $2$}  

There exist some strong reasons why we define poly-Bernoulli numbers with level $2$ as in (\ref{def:pbel2}). Poly-Cauchy numbers with level $2$ can be expressed in terms of poly-Bernoulli numbers with level $2$.  

\begin{theorem} 
For integers $n$ and $k$ with $n\ge 1$, 
$$
\mathfrak C_{2 n}^{(k)}=\sum_{m=1}^n\sum_{l=1}^m\frac{(-4)^{n-m}}{(2 m)!}\sttf2{n}{m}_2\sttf2{m}{l}_2\mathfrak B_{2 l}^{(k)}\,. 
$$ 
\label{th:pc2-pb2}
\end{theorem}  

\noindent 
{\it Remark.}  
Poly-Cauchy numbers can be expressed in terms of poly-Bernoulli numbers (\cite[Theorem 2.2]{KL}):  
$$
c_{n}^{(k)}=\sum_{m=1}^n\sum_{l=1}^m\frac{(-1)^{n-m}}{m!}\stif{n}{m}\stif{m}{l}B_{l}^{(k)}\,.
$$ 

\begin{proof}[Proof of Theorem \ref{th:pc2-pb2}] 
By Theorem \ref{th1} and the orthogonal relation 
$$
\sum_{l=j}^m(-1)^{l-j}\sttf2{m}{l}_s\stss2{l}{j}_s=
\begin{cases} 
1&\text{$m=j$};\\
0&\text{$m\ne j$}
\end{cases}\quad(s\ge 1)
$$ 
(\cite[Theorem 5.1]{KRV2}) together with $\sttf2{n}{0}_2=0$ ($n\ge 1$), we have 
\begin{align*}  
{\rm RHS}&=\sum_{m=1}^n\sum_{l=1}^m\frac{(-4)^{n-m}}{(2 m)!}\sttf2{n}{m}_2\sttf2{m}{l}_2\sum_{j=0}^l\stss2{l}{j}_2\frac{(-1)^{l-j}(2 j)!}{(2 j+1)^k}\\
&=\sum_{m=1}^n\frac{(-4)^{n-m}}{(2 m)!}\sttf2{n}{m}_2\sum_{j=0}^m\frac{(2 j)!}{(2 j+1)^k}\sum_{l=j}^m(-1)^{l-j}\sttf2{m}{l}_2\stss2{l}{j}_2\\ 
&=\sum_{m=1}^n\frac{(-4)^{n-m}}{(2 m)!}\sttf2{n}{m}_2\frac{(2 m)!}{(2 m+1)^k}\\
&=\sum_{m=1}^n\frac{(-4)^{n-m}}{(2 m+1)^k}\sttf2{n}{m}_2\\
&=\mathfrak C_{2 n}^{(k)}\quad(\cite[{\rm Theorem~1}]{Ko20})\,. 
\end{align*}
\end{proof}

On the contrary, poly-Bernoulli numbers can be expressed in terms of poly-Cauchy numbers (\cite{Ko1,KL}):
$$ 
B_{n}^{(k)}=\sum_{m=1}^n\sum_{l=1}^m(-1)^{n-m}m!\sts{n}{m}\sts{m}{l}c_{l}^{(k)}\,.
$$ 
Similarly, poly-Bernoulli numbers with level $2$ can be expressed in terms of poly-Cauchy numbers with level $2$. 

\begin{theorem} 
For integers $n$ and $k$ with $n\ge 1$, 
$$
\mathfrak B_{2 n}^{(k)}=\sum_{m=1}^n\sum_{l=1}^m(-1)^{n-m}4^{m-l}(2 m)!\stss2{n}{m}_2\stss2{m}{l}_2\mathfrak C_{2 l}^{(k)}\,. 
$$ 
\label{th:pb2-pc2}
\end{theorem} 
\begin{proof}  
Using another orthogonal relation 
$$
\sum_{l=j}^m(-1)^{l-j}\stss2{m}{l}_s\sttf2{l}{j}_s=
\begin{cases} 
1&\text{$m=j$};\\
0&\text{$m\ne j$}
\end{cases}\quad(s\ge 1)
$$ 
(\cite[Theorem 5.1]{KRV2}), we have 
\begin{align*}  
{\rm RHS}&=\sum_{m=1}^n(-1)^{n-m}(2 m)!\stss2{n}{m}_2\sum_{l=0}^m 4^{m-l}\stss2{m}{l}_2\sum_{j=1}^l\frac{(-4)^{l-j}}{(2 j+1)^k}\sttf2{l}{j}_2\\ 
&=\sum_{m=1}^n(-1)^{n-m}(2 m)!\stss2{n}{m}_2\sum_{j=0}^m\frac{4^{m-j}}{(2 j+1)^k}\sum_{l=j}^m(-1)^{l-j}\stss2{m}{l}_2\sttf2{l}{j}_2\\ 
&=\sum_{m=1}^n\frac{(-1)^{n-m}(2 m)!}{(2 m+1)^k}\stss2{n}{m}_2\\
&=\mathfrak B_{2 n}^{(k)}\,. 
\end{align*}   
\end{proof}

Other relations with Stirling numbers with level $2$ are given as follows.   

\begin{theorem}  
For $n\ge 1$,  
\begin{align} 
\frac{1}{(2 n)!}\sum_{m=0}^n\sttf2{n}{m}_2\mathfrak B_{2 m}^{(k)}&=\frac{1}{(2 n+1)^k}\,,
\label{b-sum}\\ 
\sum_{m=0}^n\stss2{n}{m}_2 4^{n-m}\mathfrak C_{2 m}^{(k)}&=\frac{1}{(2 n+1)^k}\,. 
\label{c-sum}
\end{align}
\label{th:330} 
\end{theorem} 

\noindent 
{\it Remark.}  
For poly-Bernoulli and poly-Cauchy numbers (\cite[Theorem 3]{Ko1}), we have 
\begin{align*} 
\frac{1}{n!}\sum_{m=0}^n\stif{n}{m}B_{m}^{(k)}&=\frac{1}{(n+1)^k}\,,\\ 
\sum_{m=0}^n\sts{n}{m}c_{m}^{(k)}&=\frac{1}{(n+1)^k}\,. 
\end{align*}

\begin{proof}[Proof of Theorem \ref{th:330}] 
By the orthogonal relations, we have 
\begin{align*}  
\frac{1}{(2 n)!}\sum_{m=0}^n\sttf2{n}{m}_2\mathfrak B_{2 m}^{(k)}&=\frac{1}{(2 n)!}\sum_{m=0}^n\sttf2{n}{m}_2\sum_{l=0}^m\stss2{m}{l}_2\frac{(-1)^{m-l}(2 l)!}{(2 l+1)^k}\\
&=\frac{1}{(2 n)!}\sum_{l=0}^n\frac{(2 l)!}{(2 l+1)^k}\sum_{m=l}^n(-1)^{m-l}\sttf2{n}{m}_2\stss2{m}{l}_2\\
&=\frac{1}{(2 n)!}\frac{(2 n)!}{(2 n+1)^k}=\frac{1}{(2 n+1)^k}
\end{align*}
and 
\begin{align*}  
\sum_{m=0}^n\stss2{n}{m}_2 4^{n-m}\mathfrak C_{2 m}^{(k)}&=\sum_{m=0}^n\stss2{n}{m}_2 4^{n-m}\sum_{l=0}^m\sttf2{m}{l}_2\frac{(-4)^{m-l}}{(2 l+1)^k}\\ 
&=\sum_{l=0}^n\frac{4^n}{(2 l+1)^k}\left(-\frac{1}{4}\right)^l\sum_{m=l}^n(-1)^m\stss2{n}{m}_2\sttf2{m}{l}_2\\
&=\frac{4^n}{(2 n+1)^k}\left(-\frac{1}{4}\right)^n(-1)^n=\frac{1}{(2 n+1)^k}\,. 
\end{align*}
\end{proof}

\section{Double summation formula}

The poly-Bernoulli numbers satisfy the duality formula $B_n^{(-k)}=B_k^{(-n)}$ for $n,k\ge 1$, because of the symmetric formula 
$$
\sum_{n=0}^\infty\sum_{k=0}^\infty B_n^{(-k)}\frac{x^n}{n!}\frac{y^k}{k!}=\frac{e^{x+y}}{e^x+e^y-e^{x+y}}\,. 
$$ 
Though the corresponding duality formula does not always hold for other cases, we still have the double summation formula for poly-Bernoulli numbers with level $2$. 

\begin{theorem}  
$$
\sum_{n=0}^\infty\sum_{k=0}^\infty\mathfrak B_{2 n}^{(-2 k)}\frac{x^{2 n}}{(2 n)!}\frac{y^{2 k}}{(2 k)!}=\frac{\cos x\cosh y}{2(1-\cos x)(1-\cosh 2 y)+\cos^2 x}\,. 
$$ 
\label{th:doublesum}
\end{theorem}
\begin{proof}  
We have 
\begin{align*}  
&\sum_{n=0}^\infty\sum_{k=0}^\infty\mathfrak B_{2 n}^{(-2 k)}\frac{x^{2 n}}{(2 n)!}\frac{y^{2 k}}{(2 k)!}\\
&=\sum_{k=0}^\infty\sum_{m=0}^\infty\left(2\sin\frac{x}{2}\right)^{2 m}(2 m+1)^{2 k}\frac{y^{2 k}}{(2 k)!}\\
&=\sum_{m=0}^\infty\left(2\sin\frac{x}{2}\right)^{2 m}\cosh\bigl((2 m+1)y\bigr)\\ 
&=\frac{e^y}{2}\sum_{m=0}^\infty\left(2\sin\frac{x}{2}\right)^{2 m}e^{2 m y}+\frac{e^{-y}}{2}\sum_{m=0}^\infty\left(2\sin\frac{x}{2}\right)^{2 m}e^{-2 m y}\\
&=\frac{e^y}{2}\frac{1}{1-\bigl(2 e^y\sin(x/2)\bigr)^2}+\frac{e^{-y}}{2}\frac{1}{1-\bigl(2 e^{-y}\sin(x/2)\bigr)^2}\\
&=\frac{1}{2}\left(\frac{e^y}{1-e^{2 y}(1-\cos x)}+\frac{e^{-y}}{1-e^{-2 y}(1-\cos x)}\right)\\
&=\frac{1}{2}\frac{e^{-y}\cos x+e^y\cos x}{2(1-\cos x)-(e^{2 y}+e^{-2 y})(1-\cos x)+\cos^2 x}\\
&=\frac{\cos x\cosh y}{2(1-\cos x)(1-\cosh 2 y)+\cos^2 x}\,. 
\end{align*}
\end{proof}

\section{Congruences}  

In this section, we shall show some congruent relations for $\mathfrak B_{n}^{(k)}$ for negative $k$.   

\begin{theorem}  
For $n,k\ge 1$, $\mathfrak B_{2 n}^{(-k)}\equiv 0\pmod 6$.  
\label{th:cong6}
\end{theorem}
\begin{proof}  
For $n\ge 1$, from Theorem \ref{th1} together with $\stss2{n}{0}_2=0$ ($n\ge 1$). 
$$ 
\mathfrak B_{2 n}^{(-k)}=\sum_{m=1}^n\stss2{n}{m}_2(-1)^{n-m}(2 m)!(2 m+1)^k\,. 
$$ 
Since $2|(2 m)!$ and $3|(2 m+1)^k$ for $m=1$, and $6|(2 m)!$ for $m\ge 2$, together with the fact that the Stirling numbers of the second kind with level $2$ are integers, we get the desired result.  
\end{proof}  

\begin{theorem}  
For $n,k\ge 1$, the values of $\mathfrak B_{2 n}^{(-k)}\pmod 5$ are given in the following table.  
\begin{center} 
\begin{tabular}{|c|ccccc|}\hline
\diagbox{$n$}{$k$}&$0$&$1$&$2$&$3$&$\pmod 4$\\\hline  
$0$&$3$&$4$&$2$&$1$&\\
$1$&$2$&$1$&$3$&$4$&\\
$\pmod 2$&&&&&\\\hline 
\end{tabular}
\end{center}
\label{th:cong5}
\end{theorem}
\begin{proof}  
In the terms of the summation expression of $\mathfrak B_{2 n}^{(-k)}$, $5|(2 m+1)^k$ for $m=2$, and $5|(2 m)!$ for $m\ge 3$. Hence, 
$$
\mathfrak B_{2 n}^{(-k)}\equiv(-1)^{n-1}2\cdot 3^k\pmod 5\,. 
$$  
Since $3^4\equiv 1\pmod 5$ by Fermat's little theorem, it is sufficient to check the cases for $k\equiv 0,1,2,3\pmod 4$. 
When $k\equiv 0\pmod 4$, 
$$
\mathfrak B_{2 n}^{(-k)}\equiv(-1)^{n-1}2\equiv
\begin{cases}
3&\text{($n\equiv 0\pmod 2$)}\\ 
2&\text{($n\equiv 1\pmod 2$)}
\end{cases}\pmod 5\,. 
$$ 
When $k\equiv 1\pmod 4$, 
$$
\mathfrak B_{2 n}^{(-k)}\equiv(-1)^{n-1}6\equiv
\begin{cases}
4&\text{($n\equiv 0\pmod 2$)}\\ 
1&\text{($n\equiv 1\pmod 2$)}
\end{cases}\pmod 5\,. 
$$ 
When $k\equiv 2\pmod 4$, 
$$
\mathfrak B_{2 n}^{(-k)}\equiv(-1)^{n-1}18\equiv
\begin{cases}
2&\text{($n\equiv 0\pmod 2$)}\\ 
3&\text{($n\equiv 1\pmod 2$)}
\end{cases}\pmod 5\,. 
$$ 
When $k\equiv 3\pmod 4$, 
$$
\mathfrak B_{2 n}^{(-k)}\equiv(-1)^{n-1}54\equiv
\begin{cases}
1&\text{($n\equiv 0\pmod 2$)}\\ 
4&\text{($n\equiv 1\pmod 2$)}
\end{cases}\pmod 5\,. 
$$ 
\end{proof}

\begin{theorem}  
For $n,k\ge 1$, the values of $\mathfrak B_{2 n}^{(-k)}\pmod 7$ are given in the following table.  
\begin{center} 
\begin{tabular}{|c|ccccccc|}\hline
\diagbox{$n$}{$k$}&$0$&$1$&$2$&$3$&$4$&$5$&$\pmod 6$\\\hline  
$0$&$6$&$6$&$0$&$1$&$1$&$0$&\\
$1$&$2$&$6$&$4$&$5$&$1$&$3$&\\
$2$&$1$&$2$&$1$&$6$&$5$&$6$&\\
$3$&$1$&$1$&$0$&$6$&$6$&$0$&\\
$4$&$5$&$1$&$3$&$2$&$6$&$4$&\\
$5$&$6$&$5$&$6$&$1$&$2$&$1$&\\
$\pmod 6$&&&&&&&\\\hline 
\end{tabular}
\end{center}
\label{th:cong7}
\end{theorem}
\begin{proof} 
In the terms of the summation expression of $\mathfrak B_{2 n}^{(-k)}$, $7|(2 m+1)^k$ for $m=3$, and $7|(2 m)!$ for $m\ge 4$. Since 
$$
\stss2{n}{2}_2=\frac{4^{n-1}-1}{3}\,,
$$  
\begin{align*}
\mathfrak B_{2 n}^{(-k)}&\equiv(-1)^{n-1}2\cdot 3^k+(-1)^n\frac{4^{n-1}-1}{3}4!\cdot 5^k\\
&=(-1)^{n-1}2\cdot 3^k+(-1)^n(4^{n-1}-1)5^k\pmod 7\,. 
\end{align*} 
Since $a^6\equiv 1\pmod 7$ ($a=3,4,5$) by Fermat's little theorem, it is sufficient to check the cases for $k\equiv 0,1,2,3,4,5\pmod 6$ and $n\equiv 0,1,2,3,4,5\pmod 6$. 
When $n\equiv 2\pmod 6$, 
\begin{align*}
\mathfrak B_{2 n}^{(-k)}&\equiv-2\cdot 3^k+3\cdot 5^k\\
&\equiv\begin{cases} 
-2+3=1&\text{($k\equiv 0\pmod 6$)}\\ 
-2\cdot 3+3\cdot 5\equiv -6+1\equiv 2&\text{($k\equiv 1\pmod 6$)}\\ 
-6\cdot 3+1\cdot 5\equiv -4+5=1&\text{($k\equiv 2\pmod 6$)}\\ 
-4\cdot 3+5\cdot 5\equiv -5+4\equiv 6&\text{($k\equiv 3\pmod 6$)}\\ 
-5\cdot 3+4\cdot 5\equiv -1+6=5&\text{($k\equiv 4\pmod 6$)}\\ 
-1\cdot 3+6\cdot 5\equiv -3+2\equiv 6&\text{($k\equiv 5\pmod 6$)}\\ 
\end{cases}\pmod 7\,. 
\end{align*}
Other cases are similarly proved and omitted.  
\end{proof}

\section{Bernoulli numbers with level $2$}  

When $k=1$, {\it Bernoulli numbers} $\mathfrak B_n=\mathfrak B_n^{(1)}$ {\it with level $2$} are given by the generating function 
\begin{equation}  
\frac{1}{4\sin(x/2)}\log\frac{1+2\sin(x/2)}{1-2\sin(x/2)}=\sum_{n=0}^\infty\mathfrak B_n\frac{x^n}{n!}\,.   
\label{def:bel2}   
\end{equation}

First several values of Bernoulli numbers with level $2$ are 
\begin{multline*}  
\{\mathfrak B_{2 n}\}_{0\le n\le 10}=1,\frac{2}{3}, \frac{62}{15}, \frac{1670}{21}, \frac{47102}{15}, \frac{6936718}{33}, \frac{29167388522}{1365}, \frac{9208191626}{3}, \\
\frac{150996747969694}{255}, \frac{58943788779804242}{399}, \frac{7637588708954836042}{165}\,. 
\end{multline*}

\begin{table}[h] 
\begin{center} 
\caption{Fractional parts of Bernoulli numbers with level $2$} 
\begin{tabular}{|c||ccccccccccc|}\hline
$n$&$0$&$2$&$4$&$6$&$8$&$10$&$12$&$14$&$16$&$18$&$20$\\\hline  
$\mathfrak B_{n}\mod 1$&$0$&$\frac{2}{3}$&$\frac{2}{15}$&$\frac{11}{21}$&$\frac{2}{15}$&$\frac{19}{33}$&$\frac{272}{1365}$&$\frac{2}{3}$&$\frac{19}{255}$&$\frac{188}{399}$&$\frac{37}{165}$\\\hline   
\end{tabular}
\end{center}
\end{table} 

Note that the sequence of the denominators is the same as those of cosecant numbers 
$$
-2(2^{2 n-1}-1)B_{2^n}
$$ 
(\cite{Norlund,Riordan}). 

The von Staudt-Clausen theorem \cite{Clausen,vStaudt} states that for every $n>0$, 
$$
B_{2 n}+\sum_{(p-1)|2 n}\frac{1}{p} 
$$ 
is an integer. The sum extends over all primes $p$ for which $p-1$ divides $2 n$. In \cite{Ko20a,OS,Sanchez-Peregrino}, von Staudt-Clausen's type formulas for poly-Euler numbers, Euler numbers of the second kind and poly-Bernoulli numbers have been shown.  

We can also determine the denominators of Bernoulli numbers with level $2$ completely.  

\begin{theorem}  
For every $n>0$, 
$$
\mathfrak B_{2 n}+\sum_{(p-1)|2 n}\frac{(-1)^{n-\frac{p-1}{2}}}{p} 
$$ 
is an integer. The sum extends over all odd primes $p$ for which $p-1$ divides $2 n$. 
\label{th:vs-c}
\end{theorem}  

\noindent 
{\it Examples.}  
For $n=8$ and $n=9$, 
\begin{align*} 
(\mathfrak B_{16}\mod 1)-\frac{1}{3}+\frac{1}{5}+\frac{1}{17}&=0\,,\\
(\mathfrak B_{18}\mod 1)+\frac{1}{3}+\frac{1}{7}+\frac{1}{19}&=1\,. 
\end{align*}

\begin{proof}[Proof of Theorem \ref{th:vs-c}]
From Theorem \ref{th1} and (\ref{ex:32}), notice that 
\begin{align*}  
\mathfrak B_{2 n}^{(1)}&=\sum_{m=0}^n\stss2{n}{m}_2\frac{(-1)^{n-m}(2 m)!}{2 m+1}\\
&=\sum_{m=0}^n\frac{2(-1)^{n}}{2 m+1}\sum_{j=1}^m(-1)^{j}\binom{2 m}{m-j}j^{2 n}\,.
\end{align*}
For $n\ge 1$, we see $m\ge 1$ in the above summations. 

\noindent 
{\bf Case 1.} When $2 m+1$ is composite, as $m\ge 4$, $(2 m+1)|(2 m)!$. Since $\stss2{n}{m}_2$ is an integer, every such a term of 
$$
\stss2{n}{m}_2\frac{(-1)^{n-m}(2 m)!}{2 m+1}
$$ 
is an integer.  

\noindent 
{\bf Case 2.} When $2 m+1=p$ is prime, $p\ge 3$ and $m=(p-1)/2$ is an integer.  
Now, consider the summation 
$$
\sum_{j=1}^{(p-1)/2}(-1)^{j}\binom{p-1}{(p-1)/2-j}j^{2 n}\,. 
$$ 

\noindent 
{\bf Case 2.1.} 
If $\frac{p-1}{2}|n$, then by Fermat's little theorem, 
$$
j^{2 n}\equiv 1\pmod p\quad(j=1,2,\dots,p-1)\,. 
$$ 
Hence, 
\begin{align*} 
\sum_{j=1}^{(p-1)/2}(-1)^{j}\binom{p-1}{(p-1)/2-j}j^{2 n}&\equiv \sum_{j=1}^{(p-1)/2}(-1)^{j}\binom{p-1}{(p-1)/2-j}\pmod p\\
&=-\binom{p-2}{(p-1)/2}\,.
\end{align*}
The central binomial coefficient modulo prime yields 
\begin{align*} 
\binom{p-2}{k}&\equiv\frac{(-2)(-3)\cdots(-k-1)}{1\cdot 2\cdots k}=(-1)^k(k+1)\\
&=(-1)^{\frac{p-1}{2}}\frac{p+1}{2}\pmod p\quad\left(k=\frac{p-1}{2}\right)\,. 
\end{align*}
Thus, when $p\equiv 1\pmod 4$, since  
$$  
\binom{p-2}{(p-1)/2}\equiv\frac{p+1}{2}\pmod p\,, 
$$   
\begin{align*}  
\frac{2(-1)^{n}}{2 m+1}\sum_{j=1}^m(-1)^{j}\binom{2 m}{m-j}j^{2 n}&\equiv-\frac{2(-1)^{n}}{p}\frac{p+1}{2}\\ 
&\equiv-\frac{(-1)^{n}}{p}\pmod 1\,.
\end{align*}
When $p\equiv 3\pmod 4$, since 
$$  
\binom{p-2}{(p-1)/2}\equiv\frac{p-1}{2}\pmod p\,, 
$$  
\begin{align*}  
\frac{2(-1)^{n}}{2 m+1}\sum_{j=1}^m(-1)^{j}\binom{2 m}{m-j}j^{2 n}&\equiv-\frac{2(-1)^{n}}{p}\frac{p-1}{2}\\ 
&\equiv-\frac{(-1)^{n-1}}{p}\pmod 1\,. 
\end{align*}

\noindent 
{\bf Case 2.2.} 
If $\frac{p-1}{2}\nmid n$, then by Fermat's little theorem, 
$$
j^{2 n}\equiv j^{2 n-\nu(p-1)}\pmod p\quad(j=1,2,\dots,p-1) 
$$ 
for $0<2 n-\nu(p-1)<p-1$ with $\nu=\fl{(2 n)/(p-1)}$. Notice that $2 n-\nu(p-1)$ is even. 
Hence, 
\begin{align*} 
\sum_{j=1}^{(p-1)/2}(-1)^{j}\binom{p-1}{(p-1)/2-j}j^{2 n}&\equiv \sum_{j=1}^{(p-1)/2}(-1)^{j}\binom{p-1}{(p-1)/2-j}j^{2 n-\nu(p-1)}\\
&\equiv 0\pmod p\,.
\end{align*}
\end{proof}

\section{Open problems}  

In \cite{KRV1,KRV2}, Stirling numbers of both kinds with higher level are discussed. One may wonder if poly-Bernoulli numbers with level $3$ or higher can be introduced by using the Stirling numbers with level $3$ or higher. However, the situation becomes very complicated for the case with level $3$ or higher. 

\section*{Acknowledgements} 

This work was mainly done when the author was in Tokyo in 2020. He would like to thank Professor Taku Ishii of Seikei University and Professor Masatoshi Suzuki of Tokyo Institute of Technology. Under difficult circumstances it was very hard to accomplish this work without their support. 
The author thanks the anonymous referee for careful reading of the manuscript and helpful comments and suggestions. 




\end{document}